\documentclass[11pt,a4paper]{amsart}

\usepackage{amsmath, amssymb, amsthm}
\usepackage[english]{babel}
\usepackage[all]{xy}
\usepackage{graphicx}
\usepackage{tikz}
\usetikzlibrary{matrix}
\usepackage{verbatim}
\usepackage{enumerate}
\usepackage{url}
\usepackage{bm}

\newtheorem{thm}{Theorem}[section]
\newtheorem{lem}[thm]{Lemma}

\theoremstyle{definition}
\newtheorem{defn}[thm]{Definition}
\theoremstyle{remark}
\newtheorem{rem}[thm]{Remark}

\newcommand{\imagescaling}{0.9}
\newcommand{\note}[1]{}

\usepackage{xcolor}

\newcommand{\maxdim}{N}
\newcommand{\maxpop}{M}

\DeclareMathOperator{\Fix}{Fix}

\DeclareMathOperator{\grad}{grad}

\newcommand{\Splay}{\Dp}
\newcommand{\Sync}{\Sp}

\newcommand{\vth}{\vartheta}

\newcommand{\itw}[1]{#1_2}
\newcommand{\ifo}[1]{#1_4}
\newcommand{\iel}[1]{#1_\ell}
\newcommand{\mdm}[1]{\hat{#1}}
\newcommand{\cpl}{g}

\newcommand{\aaa}{\alpha}
\newcommand{\at}{\itw{\aaa}}
\newcommand{\af}{\ifo{\aaa}}

\newcommand{\gt}{\itw{\cpl}}
\newcommand{\gtm}{\itw{\mdm{\cpl}}}
\newcommand{\gf}{\ifo{\cpl}}
\newcommand{\gfm}{\ifo{\mdm{\cpl}}}
\newcommand{\gl}{\iel{\cpl}}
\newcommand{\glm}{\iel{\mdm{\cpl}}}

\newcommand{\ud}{\mathrm{d}}

\DeclareMathOperator{\Rep}{Re}

\renewcommand{\Re}{\Rep}

\newcommand{\sm}{\smallsetminus}

\newcommand{\R}{\mathbb{R}}

\newcommand{\Z}{\mathbb{Z}}

\newcommand{\Zm}{\Z_\maxpop}

\newcommand{\Tor}{\mathbf{T}}
\newcommand{\Torn}{\Tor^\maxdim}

\newcommand{\Tormn}{\Tor^{\maxpop\maxdim}}

\newcommand{\g}{\gamma}
\newcommand{\G}{\Gamma}

\newcommand{\Ss}{\mathbf{S}}
\newcommand{\Sk}[1]{\Ss_{#1}}
\newcommand{\Sn}{\Ss_\maxdim}

\newcommand{\abs}[1]{\left|#1\right|}

\newcommand{\norm}[1]{\left\|#1\right\|}

\newcommand{\rset}[2]{\left\lbrace\, #1\,\left|\;#2\right.\right\rbrace}
\newcommand{\lset}[2]{\left\lbrace\left. #1\;\right|\,#2\,\right\rbrace}
\newcommand{\set}[2]{\rset{#1}{#2}}
\newcommand{\tset}[2]{\big\lbrace #1\,\big|\;#2\big\rbrace}
\newcommand{\sset}[1]{\left\lbrace #1\right\rbrace}

\DeclareMathOperator{\id}{id}

\newcommand{\Sp}{\mathrm{S}}
\newcommand{\Dp}{\mathrm{D}}
\newcommand{\DDD}{{\Dp\Dp\Dp}}
\newcommand{\SSS}{{\Sp\Sp\Sp}}
\newcommand{\DSS}{{\Dp\Sp\Sp}}
\newcommand{\SDS}{{\Sp\Dp\Sp}}
\newcommand{\SSD}{{\Sp\Sp\Dp}}
\newcommand{\DDS}{{\Dp\Dp\Sp}}
\newcommand{\SDD}{{\Sp\Dp\Dp}}
\newcommand{\DSD}{{\Dp\Sp\Dp}}

\newcommand{\DpS}{{\Dp\psi\Sp}}
\newcommand{\pDS}{{\psi\Dp\Sp}}
\newcommand{\SDp}{{\Sp\Dp\psi}}
\newcommand{\SpD}{{\Sp\psi\Dp}}
\newcommand{\pSD}{{\psi\Sp\Dp}}
\newcommand{\DSp}{{\Dp\Sp\psi}}

\newcommand{\pSS}{{\psi\Sp\Sp}}
\newcommand{\SpS}{{\Sp\psi\Sp}}
\newcommand{\SSp}{{\Sp\Sp\psi}}

\newcommand{\pDD}{{\psi\Dp\Dp}}
\newcommand{\DpD}{{\Dp\psi\Dp}}
\newcommand{\DDp}{{\Dp\psi\Dp}}

\newcommand{\Spp}{{\Sp\psi_2\psi_3}}
\newcommand{\pDp}{{\psi_1\Dp\psi_3}}

\newcommand{\K}{K}

\newcommand{\Cyc}{\textsf{\textbf{C}}}

\newcommand{\Wu}{W^\textnormal{u}}
\newcommand{\Ws}{W^\textnormal{s}}

\usepackage{etoolbox}%
\apptocmd{\thebibliography}{\footnotesize}{}{}%

\begin{document}



\title[Heteroclinic Cycles of Localized Frequency Synchrony]{Heteroclinic Dynamics of Localized Frequency Synchrony: Heteroclinic Cycles for Small Populations}
\author{Christian Bick}%

\address{Centre for Systems Dynamics and Control and Department of Mathematics, University of Exeter, EX4~4QF, UK}
\date{\today}

\begin{abstract}
Many real-world systems can be modeled as networks of interacting oscillatory units.  Collective dynamics that are of functional relevance for the oscillator network, such as switching between metastable states, arise through the interplay of network structure and interaction. Here, we give results for small networks on the existence of heteroclinic cycles between dynamically invariant sets on which the oscillators show localized frequency synchrony. Trajectories near these heteroclinic cycles will exhibit sequential switching of localized frequency synchrony: a population oscillators in the network will oscillate faster (or slower) than others and which population has this property changes sequentially over time. Since we give explicit conditions on the system parameters for such dynamics to arise, our results give insights into how network structure and interactions (which include higher-order interactions between oscillators) facilitate heteroclinic switching between localized frequency synchrony. 
\end{abstract}

\maketitle

\section{Introduction}

Networks of interacting oscillatory units can give rise to dynamics where the system appears to be in one metastable state before ``switching'' to another in a rapid transition. Such dynamics are in particular believed to be of functional relevance for neuronal networks where one observes sequential switching between patterns involving localized activity or synchrony~\cite{AshwinTimme2005, Britz2010, Tognoli2014}. One approach is to capture these dynamics on a macroscopic scale: one assigns each pattern an activity variable whose dynamics are then described by kinetic equations~\cite{Rabinovich2006}. The resulting equations are of generalized Lotka--Volterra type which support stable heteroclinic cycles, that is, cycles of hyperbolic equilibria which are connected by heteroclinic trajectories. The dynamics near such heteroclinic cycles now resemble sequential switching dynamics of activity patterns. Indeed, heteroclinic cycles and networks have been long studied in their own right; see~\cite{Ashwin2017} for a recent review.

However, such a qualitative approach fails to capture the dynamics on the level of single, nonlinearly interacting oscillators. In particular, is does not necessarily illuminate what ingredients of network topology and the interactions between oscillators~\cite{Stankovski2017} facilitate switching dynamics. If one assumes weak coupling, phase reduction provide a powerful tool to describe the dynamics of an oscillator network; in this reduction, each oscillator is represented by a single phase variable on the torus $\Tor := \R/2\pi\Z$ and the dynamics of the phases are described by a phase oscillator network. Simple networks of globally and identically coupled phase identical oscillators support heteroclinic cycles and networks~\cite{Ashwin2007}. The equilibria involved in these cycles are phase-locking patterns with oscillators in different clusters which have a constant phase difference. The symmetry properties of these networks, however, imply that all oscillators rotate with the same speed (frequency) on average---the network is globally frequency synchronized. In a neural network, this corresponds to all neurons to fire at the same average rate while the exact timing of firing changes.

By contrast, even networks of identical phase oscillators that are organized into different populations can give rise to dynamics where frequency synchrony is local to a population rather than global across the whole network. In other words, the interactions in a network of identical oscillators cause some units to evolve faster (or slower) than others. Dynamically invariant sets with this property  relate to ``chimeras''~\cite{Panaggio2015, Scholl2016, Omelchenko2018} who have---as patterns withe localized frequency synchrony---been hypothesized to play a functional role in the context of neuroscience~\cite{Shanahan2010, Tognoli2014, Bick2014a}. From a mathematical point of view, the notion of a weak chimera~\cite{Ashwin2014a, Bick2015c, Bick2015d} formalizes the definition of a dynamically invariant set with localized frequency synchrony for finite networks of identical phase oscillators.

Here we prove the existence of robust heteroclinic cycles between invariant sets with localized frequency synchrony in small phase oscillator networks with higher-order interactions. In contrast to attracting sets with localized frequency synchrony, the dynamics here induce sequential switching dynamics: which population of oscillators oscillates at a faster (or slower) rate will change over time. These results are of interest from several distinct perspectives. First, they illuminate how the interplay of network structure and functional interactions between units gives rise to heteroclinic dynamics in phase oscillator networks: we explicitly relate the network coupling parameters to the existence of heteroclinic cycles. Second, the results highlight how higher-order network interactions shape the (global) network dynamics; apart from higher harmonics in the phase interaction function, the higher-order interactions also include nonadditive interactions between oscillator phases which arise naturally in phase reductions of generically coupled identical oscillators~\cite{Ashwin2015a} or other resonant interactions~\cite{Komarov2013a}. Here, the interplay of higher-order interactions and nontrivial network topology and induces dynamics beyond (full) synchrony. Third, our results provide new examples of heteroclinic cycles in network dynamical systems relevant for applications. We highlight how these examples are distinct from situations previously considered in the literature.

This work is organized as follows. In this paper, we build on results in a recent brief communication~\cite{Bick2017c} to prove the existence of robust heteroclinic cycles between localized frequency synchrony; in a companion paper~\cite{Bick2018} we give a detailed discussion of the stability of such heteroclinic cycles (which may be embedded into larger heteroclinic structures). The remainder of this paper is organized as follows. In Section~\ref{sec:Prelim} we review some preliminaries on heteroclinic cycles and phase oscillator networks. In Section~\ref{sec:HetCycles3x2} we show existence of a heteroclinic cycle between localized patterns of frequency synchrony in networks consisting of three populations of two oscillators. In Section~\ref{sec:HetCycles3x3} we consider networks which consist of three populations of three oscillators and show the existence of a heteroclinic cycle of localized frequency synchrony; here, there are continua of saddle connections in two-dimensional invariant subspaces. Finally, in Section~\ref{sec:Numerics}, we give some numerical evidence that these phenomena persist in networks with more generic interactions before giving some concluding remarks.

\newcommand{\CS}{\K}
\newcommand{\HC}[2]{[#1\to#2]}

\section{Preliminaries}
\label{sec:Prelim}

\newcommand{\Mfld}{\mathcal{M}}

\subsection{Heteroclinic cycles}

Let~$\Mfld$ be a smooth $d$-dimensional manifold and let~$X$ be a smooth vector field on~$\Mfld$. For a hyperbolic equilibrium $\xi\in\Mfld$ let~$\Ws(\xi)$ and~$\Wu(\xi)$ denote its stable and unstable manifold, respectively.

\begin{defn}\label{defn:HetCycle}
A \emph{heteroclinic cycle~$\Cyc$} consists of a finite number of hyperbolic equilibria~$\xi_k\in\Mfld$, $k=1,\dotsc,Q$, together with heteroclinic trajectories
\[\HC{\xi_q}{\xi_{q+1}} \subset \Wu(\xi_q)\cap \Ws(\xi_{q+1})\neq\emptyset\]
where indices are taken modulo~$q$.
\end{defn}

For simplicity, we write $\Cyc=(\xi_1, \dotsc, \xi_Q)$. If~$\Mfld$ is a quotient of a higher-dimensional manifold and~$\Cyc$ is a heteroclinic cycle in~$\Mfld$, we also call the lift of~$\Cyc$ a heteroclinic cycle.

While heteroclinic cycles are in general a nongeneric phenomenon, they can be robust in dynamical systems with symmetry. Let~$\G$ be a finite group which acts on~$\Mfld$. For a subgroup $H\subset\G$ define the set $\Fix(H) = \set{x\in\Mfld}{\g x=x\ \forall\g\in H}$  of points fixed under~$H$; any $\Fix(H)$ is invariant under the flow generated by~$X$. For $x\in\Mfld$ let $\Gamma x=\set{\g x}{\g\in\G}$ denote its group orbit and $\Sigma(x) = \set{\g\in\G}{\g x = x}$ its \emph{isotropy subgroup}. Now assume that the smooth vector field~$X$ is $\G$-equivariant vector field on~$\Mfld$, that is, the action of the group commutes with~$X$.

Now let $\Cyc=(\xi_1, \dotsc, \xi_Q)$ be a heteroclinic cycle with the following properties. For an isotropy subgroup $\Sigma_q\subset\G$ write $P_q=\Fix(\Sigma_q)$. Now suppose that there exist~$\Sigma_q$ (and thus~$P_q$) such that $\xi_q, \xi_{q+1}\in P_q$, $\xi_{q+1}$ is a sink in~$P_q$, and $\HC{\xi_q}{\xi_{q+1}}\subset P_q$. Then~$\Cyc$ is \emph{robust} with respect to $\G$-equivariant perturbations of~$X$, that is, $\G$-equivariant vector fields close to~$X$ will have a heteroclinic cycle close to~$\Cyc$; see~\cite{Krupa1997} for details.

\subsubsection{Dissipative heteroclinic cycles}

Trajectories close to a heteroclinic network can show switching dynamics: qualitatively speaking, the trajectory will spend time close to one saddle~$\xi_q$ before a rapid transition to $\xi_{q+1}$. This is in particular the case when the heteroclinic cycle is attracting (in some sense); see for example~\cite{Ashwin2017} for a more elaborate discussion.

Here we consider a criterion where we expect attraction in some sense based on the local attraction properties at a hyperbolic equilibrium~$\xi$. Let~$\lambda^{(j)}$ denote the eigenvalues of the linearization of~$X$ at~$\xi$ ordered such that
\[
\Re\lambda^{(1)} \leq \dotsb \leq \Re\lambda^{(l)} < 0<\Re\lambda^{(l+1)}\leq \dotsb\leq\Re\lambda^{(d)}.
\]
The \emph{saddle value} (or saddle index) $\nu(\xi)=-\frac{\Re\lambda^{(l)}}{\Re\lambda^{(d)}}$ compares the rates of minimal attraction and maximal expansion close to~$\xi$; cf.~\cite{ShilnikovI, Afraimovich2016}. In particular, we say that~$\xi$ is \emph{dissipative} if $\nu(\xi)>1$. For a heteroclinic cycle $\Cyc=(\xi_1, \dotsc, \xi_Q)$, write $\nu_q:=\nu(\xi_q)$.

\begin{defn}
A heteroclinic cycle~$\Cyc$ is \emph{dissipative} if $\nu(\Cyc) := \prod_q\nu_q>1.$
\end{defn}

Intuitively speaking, a heteroclinic cycle is dissipative if there is more contraction of phase space than expansion close to the saddle points. Obviously, a heteroclinic cycle is dissipative if all its equilibria are dissipative. The following result shows that, subject to suitable additional assumptions, we may expect a dissipative heteroclinic cycle to be asymptotically stable.

\begin{rem}\label{rem:Stability}
Suppose that~$\G$ is finite and let~$\Cyc$ be a dissipative heteroclinic cycle in~$\R^d$ such that
\[\Wu(\xi_q)\sm\sset{\xi_q} \subset \bigcup_{\g\in\G}\Ws(\g\xi_{q+1}).\]
In other words, the entire unstable manifold of one saddle is contained in the stable manifold of the next saddle. Then the results in~\cite{Krupa1995} imply that~$\Cyc$ is asymptotically stable.
\end{rem}

Here we restrict ourselves to show the existence of dissipative heteroclinic cycles; we address the problem of stability explicitly in the companion paper~\cite{Bick2018}.

\subsubsection{Cyclic heteroclinic chains}

\newcommand{\Hf}{\mathfrak{H}}
\newcommand{\Gf}{\mathfrak{G}}

Definition~\ref{defn:HetCycle} of a heteroclinic cycle makes no assumptions on the number of heteroclinic trajectories between equilibria. Indeed, if there are unstable manifolds of dimension larger than one, there may be continua of heteroclinic trajectories. In such a case, the question about stability is more challenging as discussed in~\cite{Ashwin1998a}, in particular because the condition in Remark~\ref{rem:Stability} does not allow any set of points (however small) on the unstable manifold of one saddle to lie outside of the stable manifold of the next saddle.

We recall some definitions given in~\cite{Ashwin1998a}, adapted to our setting. Suppose that~$\Cyc$ is a heteroclinic cycle. We say that there is a (directed) edge between $\xi_p, \xi_q\in \Cyc$ if 
\[C_{pq}:=(\Wu(\xi_p)\cap \Ws(\xi_q ))\sm\sset{\xi_p,\xi_q}\neq \emptyset.\] 
This defines a directed graph~$\Gf(\Cyc)$. Let $\tilde\Gf(\Cyc) := \Gf(\Cyc)/\Gamma$ denote the quotient obtained by identifying vertices and connections on the same group orbits. The graph~$\tilde\Gf(\Cyc)$ is \emph{cyclic} if each vertex has unique edges entering and leaving it.

\begin{defn}
For a heteroclinic cycle~$\Cyc$ define the \emph{associated heteroclinic chain}
\[\Hf(\Cyc) = \bigcup_{(\xi_p, \xi_q)\in \Cyc^2} \Wu(\xi_p)\cap \Ws(\xi_q ).\]
If~$\tilde\Gf(\Cyc)$ is cyclic then~$\Hf(\Cyc)$ is a \emph{cyclic heteroclinic chain}.
\end{defn}

In contrast to Definition~\ref{defn:HetCycle}, the heteroclinic chain associated to a heteroclinic cycle now contain all heteroclinic trajectories which connect individual equilibria. Note that heteroclinic chains do not need to be closed in $\Mfld$: some part of~$\Wu(\xi_q)$ for some~$q$ may lie outside of the heteroclinic chain.

\subsection{Phase oscillator networks with nonpairwise interactions.}

Consider~$\maxpop$ populations of~$\maxdim$ phase oscillators where $\theta_{\sigma,k}\in\Tor$ denotes the phase of oscillator~$k$ in population~$\sigma$. Hence, the state of the oscillator network is determined by $\theta=(\theta_{1}, \dotsc, \theta_{\maxpop})\in\Tor^{\maxdim\maxpop}$ where $\theta_\sigma = (\theta_{\sigma, 1}, \dotsc, \theta_{\sigma, \maxdim})\in\Torn$ is the state of population~$\sigma$. Let $g_2, g_4: \Tor\to\R$ be $2\pi$-periodic functions and
\begin{equation}\label{eq:NetInteractions}
G_4(\theta_\tau; \phi) = \frac{1}{\maxdim^2}\sum_{m,n=1}^\maxdim g_4(\theta_{\tau, m}-\theta_{\tau,n}+\phi).
\end{equation}
Now consider the phase oscillator network where the phases of individual oscillators evolve according to
\begin{equation}\label{eq:DynMxN}
\begin{split}
\dot\theta_{\sigma,k} &= \omega+\sum_{j\neq k}\Big(g_2(\theta_{\sigma, j}-\theta_{\sigma, k})
-\K^{-} G_4(\theta_{\sigma-1}; \theta_{\sigma, j} - \theta_{\sigma, k})
\\&\qquad
\qquad\qquad\qquad+\K^{+} G_4(\theta_{\sigma+1}; \theta_{\sigma, j} - \theta_{\sigma, k})
\Big)
\end{split}
\end{equation}
where~$\omega$ is the intrinsic frequency of each oscillator\footnote{Without loss of generality we may set~$\omega$ to any value by going into a suitable co-rotating frame.}.
For these network dynamics, the phase interactions within populations are determined by the coupling (or phase interaction) function~$g_2$ whose arguments differences of oscillator pairs. By contrast, the interactions between populations, given by~\eqref{eq:NetInteractions}, are mediated by the nonpairwise interaction function~$g_4$ whose argument is a linear combination of four of the oscillators' phases.
The parameter~$\K^{-}>0$ determines the coupling strength to the previous population whereas~$\K^{+}>0$ determines the coupling strength to the previous population.
Here we assume $\K := \K^{-} = \K^{+}>0$ for simplicity.
For $g_4 = \cos$, the equations~\eqref{eq:DynMxN} approximate the dynamics of a phase oscillator networks with mean-field mediated bifurcation parameters up to rescaling of time as outlined in~\cite{Bick2017c}.

\subsubsection{Symmetries and invariant sets.}
Let~$\Sn$ denote the symmetric group of permutations of~$\maxdim$ symbols and write $\Zm = \Z/\maxpop\Z$. For a single oscillator population, the subset
\begin{align}
\Sync &:= \set{(\phi_1, \dotsc, \phi_\maxdim)\in\Torn}{\phi_k=\phi_{k+1}}
\intertext{corresponds to phases being in full \emph{phase synchrony} and}
\Splay &:= \set{(\phi_1, \dotsc, \phi_\maxdim)\in\Torn}{\phi_{k+1}=\phi_{k}+\frac{2\pi}{\maxdim}}
\end{align}
denotes a \emph{splay phase} configuration---typically we call any element of the group orbit~$\Sn\Splay$ a splay phase. For a network of interacting populations, we use the shorthand notation
\begin{subequations}\label{eq:SyncSplay}
\begin{align}
\theta_1\dotsb\theta_{\sigma-1}\Sp\theta_{\sigma+1}\dotsb\theta_{\maxpop} &= \lset{\theta\in\Tormn}{\theta_\sigma\in\Sync}\\
\theta_1\dotsb\theta_{\sigma-1}\Dp\theta_{\sigma+1}\dotsb\theta_{\maxpop} &= \lset{\theta\in\Tormn}{\theta_\sigma\in\Splay}
\end{align}
\end{subequations}
to indicate that population~$\sigma$ is fully phase synchronized or in splay phase. Consequently, $\Sp\dotsb\Sp$ ($\maxpop$~times) is the set of cluster states where all populations are fully phase synchronized and $\Dp\dotsb\Dp$ the set where all populations are in splay phase.

The network interactions in~\eqref{eq:DynMxN}, which include nonpairwise coupling, induce symmetries. More precisely, the equations~\eqref{eq:DynMxN} are $(\Sn\times\Tor)^\maxpop\rtimes \Zm$-equivariant. Each copy of~$\Tor$ acts by shifting all oscillator phases of one population by a common constant while~$\Sn$ permutes its oscillators. The action of~$\Zm$ permutes the populations cyclically. These actions do not necessarily commute.

\newcommand{\Tornmo}{\Tor^{\maxdim-1}}
\newcommand{\Torm}{\Tor^{\maxpop}}

To reduce the phase-shift symmetry~$\Torm$, we rewrite~\eqref{eq:DynMxN} in terms of phase differences $\psi_{\sigma,k} := \theta_{\sigma, k+1} - \theta_{\sigma, 1}$, $k=1, \dotsc, \maxdim-1$. Hence, with $\psi_\sigma\in\Tornmo$ we also write for example $\psi_1\Sync\dotsb\Sync$ (or simply $\psi\Sync\dotsb\Sync$ if the index is obvious) to indicate that all but the first population is phase synchronized.

The symmetries yield invariant subspaces on~$\Tormn$ for the dynamics given by~\eqref{eq:DynMxN}. In particular, the~$\Sn$ permutational symmetries within each population imply that the sets~\eqref{eq:SyncSplay} are invariant~\cite{Ashwin1992}. Moreover, any set of the form $\theta_1\dotsb \theta_\maxpop$ with $\theta_k\in\sset{\Sync,\Splay}$ is an equilibrium relative to the continuous~$\Torm$ symmetry, that is, the corresponding $\psi_1\dotsb\psi_\maxpop$ is an equilibrium in the reduced dynamics.

\subsubsection{Frequencies and localized frequency synchrony}
Suppose that $\maxpop > 1$ and let~$\theta:[0,\infty)\to\Tormn$ be a solution of~\eqref{eq:DynMxN} with initial condition $\theta(0)=\theta^0$. While $\dot\theta_{\sigma,k}(t)$ is the \emph{instantaneous angular frequency} of oscillator $(\sigma, k)$, define the \emph{asymptotic average angular frequency} of oscillator $(\sigma, k)$ by \begin{equation}
\Omega_{\sigma,k}(\theta^0):=\lim_{T\to\infty}\frac{1}{T}\int_0^T\dot\theta_{\sigma,k}(t)\ud t.
\end{equation}
Here we assume that these limit exists for all oscillators but this notion can be generalized to frequency intervals; see also~\cite{Bick2015c, Bick2015d}.

\begin{defn}
A connected flow-invariant invariant set $A\subset\Tormn$ has \emph{localized frequency synchrony} if for any~$\theta^0\in A$ we have $\Omega_{\sigma,k} = \Omega_{\sigma}$ and there exist indices $\sigma\neq\tau$ such that
\begin{align}
\Omega_{\sigma} \neq \Omega_{\tau}.
\end{align}
\end{defn}

\begin{rem}
Note that a chain-recurrent set~$A$ with localized frequency synchrony is a \emph{weak chimera} as defined in~\cite{Ashwin2014a}.
\end{rem}

\begin{lem}[Theorem~1 in~\cite{Ashwin2014a}]\label{lem:FreqSync}
The system symmetries imply $\Omega_{\sigma,k} = \Omega_{\sigma,j}$.
\end{lem}

\section{Heteroclinic Cycles for Two Oscillators per Population}
\label{sec:HetCycles3x2}

\newcommand{\tGf}{\tilde G_4}

In this section, we show the existence of robust heteroclinic cycles for networks of $\maxpop=3$ populations of $\maxdim=2$ oscillators. Let $\gt, \gf:\Tor\to\R$ denote the  $2\pi$-periodic coupling functions which govern the interactions within and between populations as above. With
\begin{equation}\label{eq:NPCN2}
\tGf(\theta_\tau; \phi) = \frac{1}{4}\big(g_4(\theta_{\tau, 1}-\theta_{\tau,2}+\phi)+g_4(\theta_{\tau, 2}-\theta_{\tau,1}+\phi)\big)
\end{equation}
the network dynamics~\eqref{eq:DynMxN} can be written as
\begin{subequations}\label{eq:Dyn3x2}
\begin{align}
\begin{split}
\dot\theta_{\sigma,1} &= \omega+\gt(\theta_{\sigma,2}-\theta_{\sigma,1})
-\K\tGf(\theta_{\sigma-1}; \theta_{\sigma,2}-\theta_{\sigma,1})
\\&\qquad\qquad
+\K\tGf(\theta_{\sigma+1}; \theta_{\sigma,2}-\theta_{\sigma,1}),
\end{split}\\
\begin{split}
\dot\theta_{\sigma,2} &= \omega+\gt(\theta_{\sigma,1}-\theta_{\sigma,2})
-\K\tGf(\theta_{\sigma-1}; \theta_{\sigma,1}-\theta_{\sigma,2})
\\&\qquad\qquad
+\K\tGf(\theta_{\sigma+1}; \theta_{\sigma,1}-\theta_{\sigma,2}).
\end{split}
\end{align}
\end{subequations}
By reducing the~$\Torm$ symmetry, we obtain the dynamics of the phase-differences as 
\begin{align}\label{eq:Dyn3x2Red}
\begin{split}
\dot\psi_\sigma &=
2\gtm(\psi_\sigma)
-\frac{\K}{2}\big(
\gfm(\psi_{\sigma-1}+\psi_\sigma)
+\gfm(\psi_\sigma-\psi_{\sigma-1})\big)
 \\&\qquad
+\frac{\K}{2}\big(
\gfm(\psi_{\sigma+1}+\psi_\sigma)
+\gfm(\psi_\sigma-\psi_{\sigma+1})\big)
\end{split}
\end{align}
where $\glm(\vth) = \frac{1}{2}(\gl(-\vth)-\gl(\vth))$, $\ell\in\sset{2,4}$, are odd. For $\glm$ we have $\glm'(\vth) = -\gl'(\vth)$.

The phase space of~\eqref{eq:Dyn3x2} is organized by invariant subspaces as sketched in Figure~\ref{fig:InvSpaces3x2}. For completeness, we characterize $\SSS$, $\DDD$ before we focus on sets with localized frequency synchrony.

In the reduced dynamics~\eqref{eq:Dyn3x2Red}, both~$\SSS$ and~$\DDD$ are equilibria. The equilibrium $\SSS = (0, 0, 0)$ lies in the intersection of the invariant subspaces $\pSS$, $\SpS$, and~$\SSp$. On these subspaces, the dynamics are given by
\begin{align}
\dot\psi &= 2\gtm(\psi).
\end{align}
Thus, the linear stability of $\SSS$ is determined by the triple eigenvalue
\begin{align}
\lambda^\SSS &= -2\gt'(0)
\end{align}
which correspond to a perturbation separating the phases of one population.
Similarly, $\DDD = (\pi, \pi, \pi)$ lies in the intersection of the invariant subspaces $\pDD$, $\DpD$, and~$\DDp$. On these invariant subspaces, the dynamics are determined by
\begin{equation}\label{eq:DynPSS}
\dot\psi = 2\gtm(\psi),
\end{equation}
as well. Linearizing at $\psi=\pi$ yields
\begin{equation}
\lambda^\DDD = -2\gt'(\pi)
\end{equation}
which determines the stability of~$\DDD$. In the full system~\eqref{eq:Dyn3x2}, there are three additional zero eigenvalues for eigenvectors along the group orbit of the phase-shift symmetry. Note that the linear stability of~$\SSS$, $\DDD$ are fully determined by the pairwise coupling~$\gt$ within populations.

\newcommand{\figonepanell}[1]{\raisebox{3.3cm}{\textbf{(#1)}}\hspace{-12pt}}

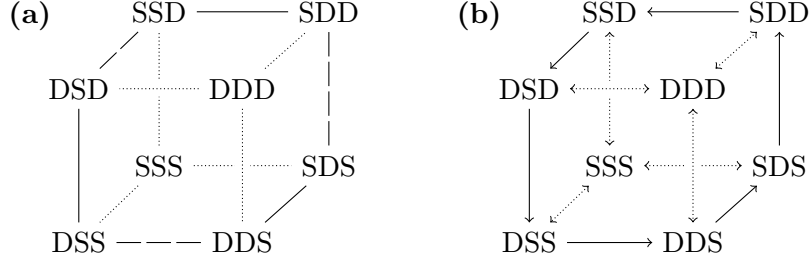
\begin{figure}
\centerline{\mbox{
\figonepanell{a}
\begin{tikzpicture}
\tikzstyle{breakd}= [dash pattern=on 23\pgflinewidth off 2pt]
  \matrix (m) [matrix of math nodes, row sep=1.3em,
    column sep=0.1em, ampersand replacement=\&]{
    \& \SSD \& \& \SDD \\
    \DSD \& \& \DDD \& \\
    \& \SSS \& \& \SDS \\
    \DSS \& \& \DDS \& \\};
  \path[-stealth]
    (m-1-2) edge [-,breakd] (m-2-1)
    (m-1-4) edge [-] (m-1-2)
    (m-2-1) edge [-] (m-4-1)
    (m-3-2) edge [-,densely dotted] (m-3-4)
            edge [-,densely dotted] (m-4-1)
            edge [-,densely dotted] (m-1-2)            
    (m-4-1) edge [-,breakd] (m-4-3)
    (m-4-3) edge [-] (m-3-4)    
    (m-3-4) edge [-,breakd] (m-1-4)
    (m-2-3) edge [-,line width=6pt,draw=white] (m-4-3)
            edge [-,line width=6pt,draw=white] (m-2-1)
            edge [-,densely dotted] (m-4-3)
            edge [-,densely dotted] (m-1-4)
            edge [-,densely dotted] (m-2-1);            
\end{tikzpicture}
}
\qquad
\figonepanell{b}
\begin{tikzpicture}
\tikzstyle{breakd}= [dash pattern=on 23\pgflinewidth off 2pt]
  \matrix (m) [matrix of math nodes, row sep=1.3em,
    column sep=0.1em, ampersand replacement=\&]{
    \& \SSD \& \& \SDD \\
    \DSD \& \& \DDD \& \\
    \& \SSS \& \& \SDS \\
    \DSS \& \& \DDS \& \\};
  \path[-stealth]
    (m-1-2) edge [->](m-2-1)
    (m-1-4) edge [->] (m-1-2)
    (m-2-1) edge [->] (m-4-1)
    (m-3-2) edge [<->,densely dotted] (m-3-4)
            edge [<->,densely dotted] (m-4-1)
            edge [<->,densely dotted] (m-1-2)            
    (m-4-1) edge [->] (m-4-3)
    (m-4-3) edge [->] (m-3-4)    
    (m-3-4) edge [->] (m-1-4)
    (m-2-3) edge [-,line width=6pt,draw=white] (m-4-3)
            edge [-,line width=6pt,draw=white] (m-2-1)
            edge [<->,densely dotted] (m-4-3)
            edge [<->,densely dotted] (m-1-4)
            edge [<->,densely dotted] (m-2-1);
\end{tikzpicture}
}
\caption{\label{fig:InvSpaces3x2}Networks of~$\maxpop=3$ oscillator populations with dynamics given by~\eqref{eq:DynMxN} give rise to relative equilibria ($\SSS, \dotsc$) which are connected by invariant subspaces (lines).
Panel~(a) indicates the dynamics on each invariant subspace: for $\maxdim=2$ the dynamics on dotted lines the dynamics are~\eqref{eq:DynPSS}, on broken lines by~\eqref{eq:DynDPS}, and on solid lines by~\eqref{eq:DynDSP}.
Panel~(b) shows a heteroclinic cycle between these equilibria.
}
\end{figure}

\subsection{Saddle invariant sets with localized frequency synchrony}

For the remainder of this section, we will consider~\eqref{eq:NPCN2} with interaction given by the coupling functions
\begin{subequations}\label{eq:CplngFuncN2}
\begin{align}
\gt(\vth) &= \sin(\vth+\at)-r\sin(2(\vth+\at)),\\
\gf(\vth) &= \sin(\vth+\af).
\end{align}
\end{subequations}
These interactions include higher harmonics; in particular, for $r>0$ the calculations above imply that both~$\SSS$ and~$\DDD$ are linearly stable. Moreover, for $\alpha:=\at=\af-\frac{\pi}{2}$ we obtain the same parametrization as in~\cite{Bick2017c, Bick2018}.

In the following we estimate the asymptotic average frequencies of $\DSS$, $\DDS$, $\SDS$, $\SDD$, $\SSD$, $\DSD$, and $\DSS$. Note that it suffices to consider $\DSS$, $\DDS$ since the latter four are their images under the~$\Zm$ action which permutes populations.

\begin{lem}\label{lem:WeakChimeras}
The sets $\DSS$, $\DDS$ and their images under the~$\Zm$ symmetry have localized frequency synchrony as subsets of~$\Tormn$ if
\begin{equation}\label{eq:CondWeakChimeras}\tag{C$\Omega$N2}
\abs{2\sin(\alpha_2)}-2\K>0.
\end{equation}
\end{lem}

\begin{proof}
By Lemma~\ref{lem:FreqSync} we have $\Omega_{\sigma} = \Omega_{\sigma,k}$ for all $k=1, \dotsc,\maxdim$, that is, all oscillators within a single population have the same asymptotic average angular frequency.

If the populations are uncoupled, $\K=0$, we have $\Omega_{1}(\theta^0) = \omega+\gt(0)$ for $\theta^0\in \Sp\psi_2\psi_3$ and $\Omega_{2}(\theta^0) = \omega+\gt(\pi)$ for $\theta^0\in \psi_1\Dp\psi_3$. This implies that for $\K\geq 0$ and coupling~\eqref{eq:CplngFuncN2} we have
$\abs{\gt(0)-\gt(\pi)}-2K = \abs{2\sin(\at)}-2K \leq \abs{\Omega_{1}(\theta^0)-\Omega_{2}(\theta^0)}$
for $\theta^0\in\Sp\Dp\psi_3$.
Consequently, $\DSS$, $\DDS$ and their symmetric counterparts have localized frequency synchrony on~$\Tormn$ if~\eqref{eq:CondWeakChimeras} is satisfied.
\end{proof}

Note that this is clearly only a sufficient condition; it suffices for our purpose but can be made better by evaluating the asymptotic average angular frequencies explicitly.

As~$\SSS$ and~$\DDD$, the sets $\DSS$, $\DDS$ are equilibria for the reduced system~\eqref{eq:Dyn3x2Red} with~$\DSS=(\pi,0,0)$, $\DDS=(\pi,\pi,0)$. On the invariant subspace~$\DpS$ we have
\begin{align}
\dot\psi &= 2\gtm(\psi) +\K(\gfm(\psi)-\gfm(\psi+\pi))\label{eq:DynDPS}
\intertext{and on $\pDS$ we have dynamics}
\dot\psi &= 2\gtm(\psi)+\K(\gfm(\psi+\pi)-\gfm(\psi)).\label{eq:DynDSP}
\end{align}
To obtain the linear stability of~$\DSS$, linearize~\eqref{eq:DynPSS} at $\psi=\pi$, \eqref{eq:DynDPS} at $\psi=0$, and~\eqref{eq:DynDSP} at $\psi=0$. With coupling~\eqref{eq:CplngFuncN2} this gives
\begin{subequations}\label{eq:StabDSS}
\begin{align}
\lambda^\DSS_1 &= 2\cos(\at)+4r\cos(2\at),\\
\lambda^\DSS_2 &= -2\K\cos(\af)-2\cos(\at)+4r\cos(2\at),\\
\lambda^\DSS_3 &= 2\K\cos(\af)-2\cos(\at)+4r\cos(2\at).
\end{align}
\end{subequations}
for the linear stability of the first, second, and third population, respectively. Similarly, linearizing~\eqref{eq:DynDSP} at $\psi=\pi$, \eqref{eq:DynDPS} at $\psi=\pi$, and \eqref{eq:DynPSS} at $\psi=0$, we obtain
\begin{subequations}\label{eq:StabDDS}
\begin{align}
\lambda^\DDS_1 &= -2\K\cos(\af)+2\cos(\at)+4r\cos(2\at),\\
\lambda^\DDS_2 &= 2\K\cos(\af)+2\cos(\at)+4r\cos(2\at),\\
\lambda^\DDS_3 &= -2\cos(\at)+4r\cos(2\at).
\end{align}
\end{subequations}
for the eigenvalues which determine the linear stability of~$\DDS$.

\subsection{Heteroclinic cycles}

In the previous section, we evaluated the local properties of the dynamically invariant sets~$\DSS$ and~$\DDS$. Heteroclinic cycles require conditions on the local stability as well as the existence of global saddle connections.

\begin{lem}\label{lem:HetCycleN2}
Suppose that
\begin{align}
\label{eq:CondSaddles}\tag{C$\lambda$N2'}
\lambda^\DSS_3&<0<\lambda^\DSS_2, &
\lambda^\DDS_2&<0<\lambda^\DDS_1.
\end{align}
Then the network~\eqref{eq:Dyn3x2} with $\maxpop=3$ populations of $\maxdim=2$ phase oscillators with dynamics~\eqref{eq:Dyn3x2} and coupling functions~\eqref{eq:CplngFuncN2} has a heteroclinic cycle
\begin{equation*}\label{eq:HetCycle1}
\Cyc_2 = (\DSS, \DDS, \SDS, \SDD, \SSD, \DSD, \DSS).
\end{equation*}
\end{lem}

\begin{proof}
We first show that there is a heteroclinic connection $\HC{\DSS}{\DDS}$. Note that~\eqref{eq:CondSaddles} implies $\Wu(\DSS),\Ws(\DDS)\subset\pDS$. To show that $\Wu(\DSS)\cap \Ws(\DDS)\neq\emptyset$, consider the dynamics on~$\DpS$: For the coupling functions~\eqref{eq:CplngFuncN2} the dynamics~\eqref{eq:DynDPS} on~$\DpS$ evaluate to
\begin{align*}
\dot\psi &= 
-2\cos(\at)\sin(\psi)
+2r\cos(2\at)\sin(2\psi)
-2\K\cos(\af)\sin(\psi)\\
&= \sin(\psi)(A_\DpS + B_\DpS\cos(\psi)),
\end{align*}
where
$A_\DpS=-2\K\cos(\af)-2\cos(\at)$,
$B_\DpS=4r\cos(2\at)$.
Since
$\lambda^\DSS_2 = A_\DpS+B_\DpS$ and $\lambda^\DDS_2 = -A_\DpS+B_\DpS$, conditions~\eqref{eq:CondSaddles} yields
\[\pm B_\DpS < A_\DpS.\]
Consequently, given~\eqref{eq:CondSaddles}, there are no equilibria on~$\DpS$ other than the point where $\sin(\psi)=0$---these correspond to~$\DSS$ and~$\DDS$. Hence there is a heteroclinic connection $\HC{\DSS}{\DDS}$.

Similarly, to show that there is $\HC{\DDS}{\SDS}$, note that by~\eqref{eq:CondSaddles} we have $\Wu(\DDS), \Ws(\SDS)\subset\pDS$. The dynamics~\eqref{eq:DynDSP} on~$\pDS$ are
\begin{align*}
\dot\psi &= 
-2\cos(\at)\sin(\psi)
+2r\cos(2\at)\sin(2\psi)
+2\K\cos(\af)\sin(\psi)\\
&= \sin(\psi)(A_\pDS + B_\pDS\cos(\psi)).
\end{align*}
By~\eqref{eq:CondSaddles} and the~$\Z_3$ symmetry we have
\[\pm B_\pDS < A_\pDS.\]
as above. This implies that there is a heteroclinic connection $\HC{\DDS}{\SDS}$ if~\eqref{eq:CondSaddles} holds.
\end{proof}

\begin{rem}
Note that here the local conditions~\eqref{eq:CondSaddles} suffice to guarantee the existence of global saddle connections. This indicates that more than two harmonics are needed for generic bifurcation behavior; cf.~\cite[Corollary~1]{Ashwin2016}. 
\end{rem}

By replacing~\eqref{eq:CondSaddles} with a stricter set of conditions we immediately obtain the following statement.

\begin{lem}
The heteroclinic cycle~$\Cyc_2$ is dissipative if
\begin{equation}
\label{eq:CondSaddlesDiss}\tag{C$\lambda$N2}
\begin{split}
\lambda^\DSS_3<\lambda^\DSS_1&<0<\lambda^\DSS_2,\\
\lambda^\DDS_2<\lambda^\DDS_3&<0<\lambda^\DDS_1,
\end{split}
\end{equation}
and
\begin{align}\label{eq:CondSaddleValues}\tag{C$\nu$N2}
\nu^\DSS=-\frac{\lambda^\DSS_1}{\lambda^\DSS_2}&>1, & \nu^\DDS=-\frac{\lambda^\DDS_3}{\lambda^\DDS_1}&>1.
\end{align}
\end{lem}

This leads to the main result of this section.

\begin{thm}\label{thm:HetCycleN2}
The network of $\maxpop=3$ populations of $\maxdim=2$ phase oscillators with dynamics~\eqref{eq:Dyn3x2} and coupling functions~\eqref{eq:CplngFuncN2} supports a robust dissipative heteroclinic cycle
\begin{equation*}\label{eq:HetCycle}
\Cyc_2 = (\DSS, \DDS, \SDS, \SDD, \SSD, \DSD, \DSS)
\end{equation*}
between dynamically invariant sets with localized frequency synchrony.
\end{thm}

\begin{proof}
Note that if \eqref{eq:CondSaddles}---or~\eqref{eq:CondSaddlesDiss}---holds, the heteroclinic trajectories in Lemma~\ref{lem:HetCycleN2} are source-sink connections in an invariant subspace forced by symmetry. Hence, to prove the assertion it suffices to show that there are indeed parameter values such that~\eqref{eq:CondWeakChimeras}, 
\eqref{eq:CondSaddlesDiss}, and
\eqref{eq:CondSaddleValues} are satisfied simultaneously.

Let $(\at, \af) = (\frac{\pi}{2}, \pi)$ and $r>0$. Set 
$\lambda^u:=\lambda^\DSS_2=\lambda^\DDS_1$, 
$\lambda^{>}:=\lambda^\DSS_1=\lambda^\DDS_3$, 
$\lambda^{\gg}:=\lambda^\DSS_3=\lambda^\DDS_2$. The stability conditions~\eqref{eq:CondSaddlesDiss} thus reduce to
\begin{subequations}
\begin{align}\label{eq:N2Stab}
\lambda^u &= -4r+2\K>0\\
\lambda^{>} &= -4r < 0\\
\lambda^{\gg} &= -4r-2\K <0
\end{align}
\end{subequations}
which are satisfied for
\begin{equation}
0<r<\frac{1}{2}\K.
\end{equation} 
Then $\lambda^u > 0 > \lambda^{>} > \lambda^{\gg}$ and the saddle values $\nu=\nu^\DSS=\nu^\DDS$ of the equilibria evaluate to
\[
\nu=-\frac{\lambda^{>}}{\lambda^u}=\frac{2r}{\K-2r}.
\]
and thus~\eqref{eq:CondSaddleValues}, that is, $\nu>1$, holds for
\begin{equation}
\frac{1}{4}\K<r
\end{equation}
Finally, for $\at=\frac{\pi}{2}$ Condition~\eqref{eq:CondWeakChimeras} is equivalent to $\K < 1$.

In summary, for $(\at, \af) = (\frac{\pi}{2}, \pi)$ the conditions~\eqref{eq:CondWeakChimeras}, 
\eqref{eq:CondSaddlesDiss}, and
\eqref{eq:CondSaddleValues}, hold simultaneously if
\begin{equation}
0< \K < 4r < 2\K < 2
\end{equation}
which completes the proof.
\end{proof}

\newcommand{\figypanell}[1]{\raisebox{3.7cm}{\textbf{(#1)}}\hspace{-16pt}}

\begin{figure}
\centerline{
\figypanell{a}
\includegraphics[scale=\imagescaling]{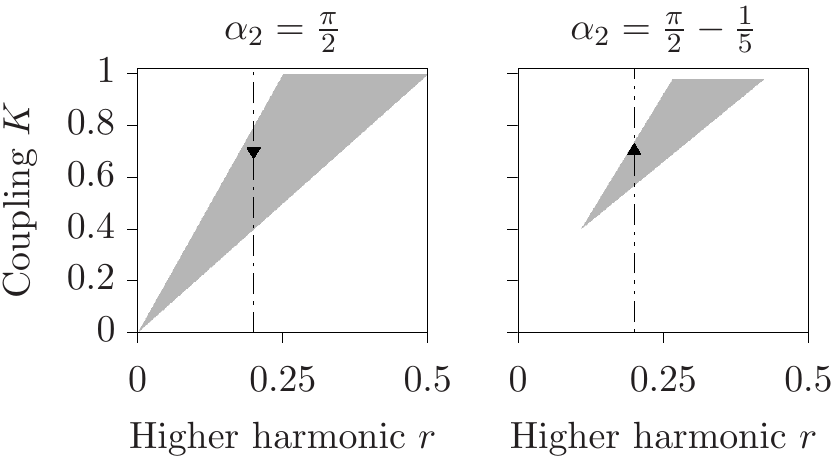}
\hfill
\figypanell{b}
\includegraphics[scale=\imagescaling]{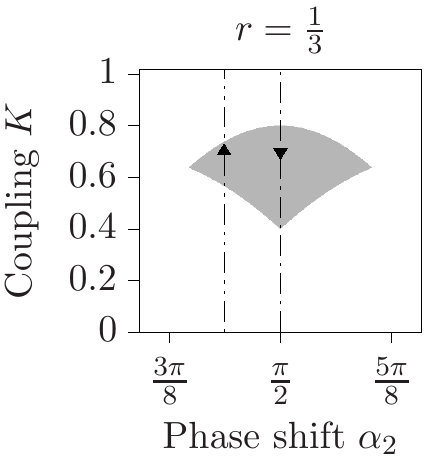}
}
\caption{\label{fig:ParamSpace}
The parameter values where Conditions \eqref{eq:CondWeakChimeras}, 
\eqref{eq:CondSaddlesDiss}, and
\eqref{eq:CondSaddleValues}
hold are contained in the shaded area; here $\af=\pi$ is fixed. Identical vertical lines indicate the same parameter values across Panels~(a) and~(b). The points $(\at, r, \K)=(\frac{\pi}{2}, \frac{1}{2}, \frac{7}{10})$ and $(\at, r, \K)=(\frac{\pi}{2}-\frac{1}{5}, \frac{1}{2}, \frac{7}{10})$ are plotted in all panels.
}
\end{figure}

Again, conditions~\eqref{eq:CondWeakChimeras}, 
\eqref{eq:CondSaddlesDiss}, and
\eqref{eq:CondSaddleValues}
for the existence of a dissipative robust heteroclinic cycle are sufficient. Figure~\ref{fig:ParamSpace} illustrates the region in parameter space that is given by these conditions.

\begin{rem}\label{rem:N2Closed}
For $\Cyc_2$ set $\xi_1 = \DSS, \xi_2=\DDS, \dotsc, \xi_6=\DSD$. We have $\Wu(\xi_q)\sm\sset{\xi_q}\subset \Ws(\xi_{q+1})$ and thus the heteroclinic chain associated with~$\Cyc_2$ is closed.
\end{rem}

\section{Heteroclinic Cycles for Three Oscillators per Population}
\label{sec:HetCycles3x3}

We now consider phase oscillator networks~\eqref{eq:DynMxN} with more than two oscillators per population. Throughout this section assume that $\maxpop=3$ and $\maxdim=3$ and suppose that the phase interaction given by the coupling functions
\begin{subequations}\label{eq:CplngFuncN3}
\begin{align}
g_2(\phi) &= \sin(\phi+\at)-r(a_2\sin(2(\phi+\at))+\sin(6(\phi+\at))),\\
g_4(\phi) &= \sin(\phi+\af).
\end{align}
\end{subequations}
If the populations are uncoupled, $\K=0$, then for $\alpha=\at=\af-\frac{\pi}{2}=\frac{\pi}{2}$ and $r>0$ there is bistability between~$\Sync$ and~$\Splay$ in each population~\cite{Ashwin1992}.

In this section, we show that there are dissipative robust heteroclinic cycles for networks of $\maxpop=3$ populations of $\maxdim=3$ oscillators with coupling~\eqref{eq:CplngFuncN3}. Indeed, we proceed as before and derive conditions for which there heteroclinic source-sink connections on invariant subspaces forced by symmetry. The resulting heteroclinic network will be robust. Thus, for the remainder of the section we set $(\at,\af)=(\frac{\pi}{2},\pi)$ rather writing down the conditions in full generality.

A stability analysis of the phase configurations~$\SSS$ and~$\DDD$ can be done as in the previous section. Moreover, we restrict ourselves to~$\DSS$ and~$\DDS$ because of symmetry.

\subsection{Local dynamics.}

We first establish conditions for~$\DSS$, $\DDS$ to be invariant sets with localized frequency synchrony which are suitable saddles in the reduced system.

\begin{lem}\label{lem:N3FreqSync}
The sets $\DSS$, $\DDS$ and their images under the~$\Zm$ symmetry have localized frequency synchrony as subsets of~$\Tormn$ if
\begin{equation}\label{eq:CondWeakChimerasN3}\tag{C$\Omega$N3}
4\K<9.
\end{equation}
\end{lem}

\begin{proof}
The proof is essentially the same as for Lemma~\ref{lem:WeakChimeras}. Frequency synchrony with populations is given by Lemma~\ref{lem:FreqSync}.

For $\K=0$ we have
$\Omega_{1}(\theta^0) = \omega+2\gt(0)=\omega+2$ for $\theta^0\in\Spp$ and 
$\Omega_{2}(\theta^0) = \omega+\gt(2\pi/3)+\gt(4\pi/3)=\omega-1$ for $\theta^0\in \pDp$.
This implies that for $\K\geq 0$ and coupling~\eqref{eq:CplngFuncN3} we have
$3-\frac{4}{3}\K \leq \abs{\Omega_{1}(\theta^0)-\Omega_{2}(\theta^0)}$
Thus $\DSS$, $\DDS$ have localized frequency synchrony on~$\Tormn$ if~\eqref{eq:CondWeakChimerasN3} is satisfied.
\end{proof}

\begin{lem}\label{lem:N3LocalStab}
In the reduced system, the equilibria~$\DSS$, $\DDS$ are hyperbolic saddles with two-dimensional unstable manifold if
\begin{equation}\label{eq:StabilityN3}\tag{C$\lambda$N3}
0<10r<\K.
\end{equation}
\end{lem}

\begin{proof}
The linearization at~$\DSS$ yields eigenvalues
\begin{subequations}\label{eq:StabDSSN3}
\begin{align}
\lambda_1^\DSS &= -15r\pm \frac{3}{2}i,\\
\lambda_2^\DSS &= -24r+3\K,\\
\lambda_3^\DSS &= -24r-3\K
\end{align}
\end{subequations}
which correspond to the stability of the phase configuration of the first, second, and third population, respectively.

Similarly, for the linearization of the vector field at~$\DDS$ we have
\begin{subequations}\label{eq:StabDDSN3}
\begin{align}
\lambda_1^\DDS &= -15r+\frac{3}{2}\K\pm \frac{3}{2}i,\\
\lambda_2^\DDS &= -15r-\frac{3}{2}\K\pm \frac{3}{2}i,\\
\lambda_3^\DDS &= -24r
\end{align}
\end{subequations}
which govern the linear stability of the phase configuration of the first, second, and third population, respectively.

Thus, we have hyperbolic saddles with two-dimensional unstable manifold if $0<r$ and
\[\min\sset{\Re(\lambda_1^\DDS), \Re(\lambda_2^\DSS)} = -15r+\frac{3}{2}\K > 0,\]
which is equivalent to~\eqref{eq:StabilityN3}, as asserted.
\end{proof}

Note that expansion at~$\DSS$ is determined by the double real eigenvalue~$\lambda_2^\DSS$ forced by symmetry. The eigenvalues of the linearization~\eqref{eq:StabDSSN3} and~\eqref{eq:StabDDSN3} also give insight into the local bifurcations as parameters are varied. For example, $\DDS$ undergoes a Hopf bifurcation as the parameter~$r$ goes through zero.

\newcommand{\CIR}{\mathcal{C}}
\newcommand{\dCIR}{\partial\CIR}
\newcommand{\CIRcl}{\overline{\CIR}}

\subsection{Global dynamics}
In the previous section we established the existence of suitable saddle invariant sets. To obtain a heteroclinic cycle, these invariant sets have to be joined by a heteroclinic trajectory.

Both~$\DSS$ and~$\DDS$ lie in the two-dimensional invariant subspace~$\DpS$ and~$\DDS$ while~$\SDS$ lie in the two-dimensional invariant subspace~$\pDS$. Because of the permutational symmetry within populations, the dynamics on both~$\DpS$ and~$\pDS$ are~$\Sk{3}$ equivariant. As discussed in the context of Lemma~\ref{lem:FreqSync}, this implies that the phase ordering within each population is preserved. Hence it suffices to consider the dynamics on (the closure of) the invariant simplex, commonly referred to as the canonical invariant region,
\[\CIR :=\set{\psi=(\psi_1, \psi_2)\in\Tor^2}{0<\psi_1<\psi_2<2\pi}\]
for the phase differences; cf.~\cite{Ashwin1992, Ashwin2016}. The phase configuration $\Sync = (0, 0)$, where all oscillators are phase synchronized, lies on the boundary~$\dCIR$ of~$\CIR$ and the splay phase configuration $\Dp = \big(\frac{2\pi}{3}, \frac{4\pi}{3}\big)$ is its centroid as illustrated in Figure~\ref{fig:DpS}(a). For a function $F:\CIR\to\R$ let~$\norm{F}_\CIR$ denote the $\sup$ norm on~$\CIR$.

The vector field for the dynamics of~\eqref{eq:DynMxN} with coupling~\eqref{eq:CplngFuncN3} and $\at=\frac{\pi}{2}$, $\af=\pi$ can now be evaluated explicitly. Write
{\allowdisplaybreaks
\begin{subequations}
\begin{align}
X_0(\psi) &= \left(\begin{array}{l}\cos(\psi_1-\psi_2)-\cos(\psi_2)\\ 
\cos(\psi_2-\psi_1)-\cos(\psi_1)\end{array}\right),\\
X_K(\psi) &= \left(\begin{array}{l}\sin(\psi_1-\psi_2)+\sin(\psi_2)+2\sin(\psi_1)\\
\sin(\psi_2-\psi_1)+2\sin(\psi_2)+\sin(\psi_1)\end{array}\right),\\
X_r(\psi) &= \left(\begin{array}{l}
\sin(6(\psi_2-\psi_1))+\sin(2(\psi_2-\psi_1))-\sin(6\psi_2)\\
\phantom{abcd}-\sin(2\psi_2)-2\sin(6\psi_1)-2\sin(2\psi_1)\\ \sin(6(\psi_1-\psi_2))+\sin(2(\psi_1-\psi_2))-2\sin(6\psi_2)\\
\phantom{abcd}-2\sin(2\psi_2)-\sin(6\psi_1)-\sin(2\psi_1)\end{array}\right).
\end{align}
\end{subequations}
}%
The dynamics on~$\DpS$ are given by
\begin{equation}\label{eq:DynN3DpS}
\dot\psi = X_0(\psi) + r X_r(\psi) + \K X_\K(\psi)
\end{equation}
and the dynamics on~$\pDS$ are given by
\begin{equation}\label{eq:DynN3pDS}
\dot\psi = X_0(\psi) + r X_r(\psi) - \K X_\K(\psi).
\end{equation}
The first term describes interaction within each population given by the first harmonic only, the second term is the intra-population interaction through higher harmonics, and the third term are interactions arising through the coupling between populations.

\begin{figure}
\centerline{
\includegraphics[width=0.28\linewidth]{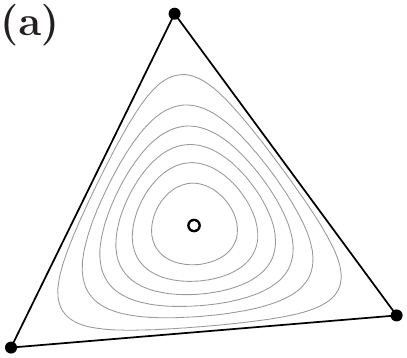}
\hfill
\includegraphics[width=0.28\linewidth]{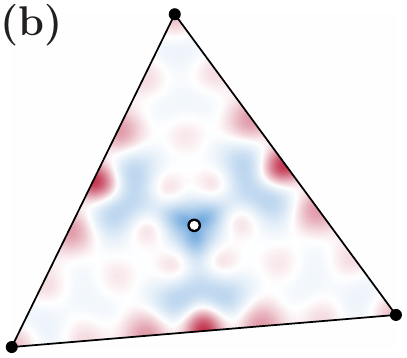}
\hfill
\includegraphics[width=0.28\linewidth]{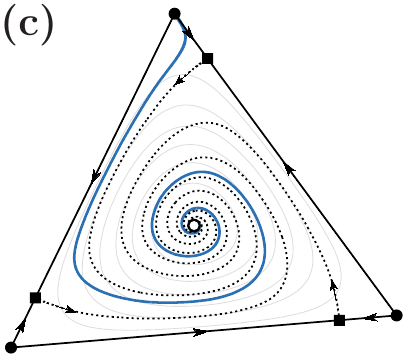}
}
\caption{\label{fig:DpS}
The dynamics on the invariant subspaces $\DpS$ and $\pDS$ are determined by the dynamics on~$\CIR$. The boundary~$\dCIR$ is composed of $\psi_1=0$, $\psi_2=0$, and $\psi_1=\psi_2$ (black lines) which intersect in~$\Sync$ (solid dot). The splay configuration~$\Splay$ (hollow dot) is the centroid of~$\CIR$.
Panel~(a) shows equipotential lines of~$V$, given by~\eqref{eq:Potential}, that are dynamically invariant if $r=\K=0$. 
Panel~(b) shows~$R(\psi)$ (red positive, blue negative) and we have $\|R(\psi)\|_\CIR<15$.
Panel~(c) shows the dynamics on $\DpS$ for $r=0.01$, $K=0.16$. There are additional equilibrium points (squares) on~$\dCIR$ whose unstable manifolds (dotted lines) lie in the stable manifold of~$\Splay$. A heteroclinic trajectory between~$\Sync$ and~$\Splay$ is depicted in blue.
}
\end{figure}

\begin{lem}\label{lem:N3SaddleConnection}
Suppose that the stability condition~\eqref{eq:StabilityN3} holds. Then there are heteroclinic trajectories $\HC{\DSS}{\DDS}$ and $\HC{\DDS}{\SDS}$ if
\begin{equation}\label{eq:HetOrbitsN3}\tag{C$\psi$N3}
15r<K.
\end{equation}
\end{lem}

\begin{proof}
First note that the dynamics of the uncoupled system, $\K=0$, without higher harmonics, $r=0$, is integrable (on both~$\DpS$ and~$\pDS$) since the quantity
\begin{equation}\label{eq:Potential}
V(\psi_1, \psi_2) = \sin\left(\frac{\psi_1}{2}\right)\sin\left(\frac{\psi_2}{2}\right)\sin\left(\frac{\psi_2-\psi_1}{2}\right)
\end{equation}
is preserved; see also~\cite{Ashwin2016}. (Some level sets of~$V$ are depicted in Figure~\ref{fig:DpS}(a).) That is, we have
\[\dot V := \langle \grad V, \dot\psi\rangle = \langle \grad V, X_0\rangle = 0.\] 
Note that~$V>0$ on~$\CIR$, it vanishes on its boundary~$\dCIR$, and takes a unique maximum~$V(\Splay)$ at~$\Splay$.

We will first consider the dynamics on~$\DpS$ and show that there is a source-sink heteroclinic trajectory $\HC{\DSS}{\DDS}$. By assumption~\eqref{eq:StabilityN3}, we have that~$\DSS$ is a source in~$\DpS$ with $\Wu(\DSS)\subset\DpS$. In the following, we derive conditions on~$\K$ and~$r$ such that~$V$ is a (Lyapunov-like) potential function on~$\CIR$ which guarantee that~$V$ is strictly increasing along trajectories in~$\CIR$. Thus, any trajectory in~$\CIR$ converges to~$\Splay\in\CIR$ which yields a heteroclinic trajectory $\HC{\DSS}{\DDS}$.

Now consider nontrivial coupling between populations, $\K>0$, while higher harmonics are absent, $r=0$. Using trigonometric identities we have
\begin{equation}
\dot V = \K\langle\grad V, X_\K\rangle = \K(V(\psi) + V(2\psi)).
\end{equation}
The potential~$V$ increases along trajectories if $\dot V>0$ on $\CIR\sm\Splay$. Rearranging terms, $\dot V>0$ is equivalent to
\begin{equation}
1 >\frac{V(2\psi)}{V(\psi)} = -2\cos(\psi_2-\psi_1)-2\cos(\psi_2)-2\cos(\psi_1)-2 =: Q(\psi)
\end{equation}
for $\psi\in\CIR\sm\Splay$. The function~$Q$ has a unique maximum on~$\CIR$ at $\psi=\Splay$ with $Q(\Splay) = 1$. Thus, $\dot V > 0$, and trajectories in $\CIR$ approach~$\Dp$ asymptotically.

We now show that this property persists for $r > 0$ sufficiently small. We have
\[\dot V = \K \langle\grad V, X_K\rangle + r \langle\grad V, X_r\rangle\]
and by the calculations above, we know that $\langle\grad V, X_K\rangle$ only vanishes for $\psi\in\partial\CIR\cup\Dp$. The condition $\dot V > 0$ on $\CIR\sm\Dp$ is equivalent to
\begin{equation}
\frac{\K}{r} > -\frac{\langle\grad V, X_r\rangle}{\langle\grad V, X_K\rangle}=:R(\psi).
\end{equation}
Note that any singularity of~$R$ is removable since $\langle\grad V, X_r\rangle$ also vanishes on $\CIR\sm\Dp$ at the same order. Hence, $R(\psi)$ is a bounded function on $\CIRcl$ and evaluating minima and maxima on~$\CIRcl$ yields $\norm{R}_\CIR<15$; cf.~Figure~\ref{fig:DpS}(b). This implies that $\dot V > 0$ on $\CIR\sm\Splay$ if
\[15r < \K.\]
which yields a heteroclinic connection between the source~$\Sync$ and the sink~$\Splay$ on~$\DpS$.

The proof that there is a robust heteroclinic trajectory~$\HC{\DDS}{\SDS}$ in~$\pDS$ is analogous. We have $\Wu(\DDS)\subset\pDS$ and a heteroclinic trajectory $\HC{\DDS}{\SDS}$ is a trajectory connecting the source~$\Splay$ and sink~$\Sync$ on~$\pDS$. Thus, it suffices to show that $\dot V < 0$ on $\CIR\sm\Splay$. Note that the vector fields~\eqref{eq:DynN3DpS} and~\eqref{eq:DynN3pDS} only differ by the sign of~$\K$. We proceed as above and in the last step the condition $\dot V < 0$ is equivalent to
\[
\frac{\K}{r} > -R(\psi).
\]
Again, $\norm{R}_\CIR<15$ implies that if
\[15r < \K\]
we have a robust heteroclinic connection between the source~$\Splay$ and the sink~$\Sync$ on~$\pDS$.
\end{proof}

In fact, Lemma~\ref{lem:N3SaddleConnection} implies the existence of a heteroclinic cycle
\begin{equation}
\Cyc_3 = (\DSS, \DDS, \SDS, \SDD, \SSD, \DSD, \DSS).
\end{equation}

\begin{lem}\label{lem:N3Dissipative}
Suppose that the assumptions of Lemma~\ref{lem:N3SaddleConnection} hold. The heteroclinic cycle~$\Cyc_3$ is dissipative if
\begin{equation}\label{eq:SaddleValN3}\tag{C$\nu$N3}
\K<18r.
\end{equation}
\end{lem}

\begin{proof}
In order to get dissipativity of the cycle, we need to control the product of the saddle values. More precisely, the cycle is dissipative if
\[\nu^\DSS\nu^\DDS = \left(\frac{15r}{3\K-24r}\right)\left(\frac{24r}{\frac{3}{2}\K-15r}\right) > 1.\]
It is straightforward to verify that this condition is equivalent to~\eqref{eq:SaddleValN3}.
\end{proof}

Note that the conditions \eqref{eq:CondWeakChimerasN3}, \eqref{eq:StabilityN3}, \eqref{eq:HetOrbitsN3}, and \eqref{eq:SaddleValN3} as given in Lemmas~\ref{lem:N3FreqSync}--\ref{lem:N3Dissipative} are satisfied simultaneously for
\begin{equation}\label{eq:HetCycN3ParCond}
0<\frac{1}{18}\K<r<\frac{1}{15}\K<
\frac{3}{20}.
\end{equation}
Moreover, the heteroclinic trajectories are source-sink connections in invariant subspaces forced by symmetry. This proves the following result.

\begin{thm}\label{thm:HetCycleN3}
The network of $\maxpop=3$ populations of $\maxdim=3$ phase oscillators with dynamics~\eqref{eq:DynMxN} and coupling functions~\eqref{eq:CplngFuncN3} has a robust dissipative heteroclinic cycle~$\Cyc_3$ between dynamically invariant sets with localized frequency synchrony.
\end{thm}

\subsection{Nonclosed heteroclinic chains}

\newcommand{\Ccl}{\Cyc^\textrm{cl}}

For networks of $\maxdim=3$ oscillators per population, there are continua of heteroclinic connections between equilibria. In particular, the heteroclinic chain associated with~$\Cyc_3$ contains all heteroclinic trajectories $\HC{\DSS}{\DDS}$, $\HC{\DDS}{\SDS}$ of trivial isotropy. While the resulting associated heteroclinic chain~$\Hf(\Cyc_3)$ is cyclic, it is not closed. In particular, the condition of Remark~\ref{rem:Stability} is not satisfied.

To formalize this observation, we adapt some terminology that was recently introduced in~\cite{Ashwin2018}.

\begin{defn}
Let~$\Cyc=(\xi_1, \dotsc, \xi_Q)$ be a heteroclinic cycle and~$\Hf(\Cyc)$ its associated heteroclinic chain. An equilibrium 
\begin{enumerate}[(i)] 
\item $\xi_q$ is \emph{complete in~$\Hf(\Cyc)$} if $\Wu(\xi_q)\subset\Hf(\Cyc)$,
\item $\xi_q$ is \emph{almost complete in~$\Hf(\Cyc)$} if $\Wu(\xi_q)\sm\Hf(\Cyc)$ is of measure zero (with respect to any Riemannian measure on $\Wu(\xi_q)$),
\item $\xi_q$ is \emph{equable in~$\Hf(\Cyc)$} if~$\dim(C_{pq})$ is equal for all~$p$ with $C_{pq}\neq\emptyset$.
\end{enumerate}
The heteroclinic chain~$\Hf(\Cyc)$ is complete, almost complete, or equable if all its equilibria are complete, almost complete, or equable respectively.
\end{defn}

Note that completeness relates to \emph{clean} heteroclinic networks defined in~\cite{Field2017}. With these notions we obtain the following statement.

\begin{thm}
For parameters satisfying~\eqref{eq:HetCycN3ParCond} and~$r$ sufficiently small, the heteroclinic chain~$\Hf(\Cyc_3)$ is cyclic, equable, and almost complete but not complete.
The closure 
of $\Hf(\Cyc_3)$ is complete, but neither cyclic nor equable.%
\end{thm}
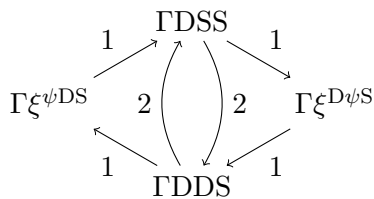
\begin{figure}
\begin{tikzpicture}
\tikzstyle{breakd}= [dash pattern=on 23\pgflinewidth off 2pt]
  \matrix (m) [matrix of math nodes, row sep=1.3em,
    column sep=1.5em, ampersand replacement=\&]{
    \& \G\DSS \& \\
    \G\xi^\pDS \& \& \G\xi^\DpS\\
    \& \G\DDS \& \\};
  \path[-stealth]
    (m-1-2) edge [->, bend left] node [right] {2} (m-3-2)  
    (m-1-2) edge [->] node [above right] {1} (m-2-3)  
    (m-2-3) edge [->] node [below right] {1} (m-3-2)
    (m-3-2) edge [->, bend left] node [left] {2} (m-1-2)  
    (m-3-2) edge [->] node [below left] {1} (m-2-1)    
    (m-2-1) edge [->] node [above left] {1} (m-1-2);    
\end{tikzpicture}
\caption{\label{fig:HetChain}The closure of the heteroclinic chain~$\Hf(\Cyc_3)$ is contained in the noncyclic heteroclinic chain~$\Hf(\Ccl)$ shown here. The labels on the arrows denote the dimension of the set of heteroclinic connections; the stability of $\xi^\pDS,\xi^\DpS$ was evaluated for $r=0.01$, $K=0.16$.
The graph $\tilde\Gf(\Cyc_3)$ corresponds to the edges corresponding to two-dimensional continua of heteroclinic trajectories.}
\end{figure}%
\begin{proof}
Note that the heteroclinic chain~$\Hf(\Cyc_3)$ associated with~$\Cyc_3$ is cyclic---because $\Wu(\DSS)\subset\DpS$ and $\Wu(\DDS)\subset\pDS$---and thus equable.

First, consider the invariant set~$\DpS$; it suffices to consider~$\overline{\CIR}$ as above. The invariant set~$\dCIR$ consists of points of nontrivial isotropy, i.e., $\Sigma(\theta)\neq\sset{\id}$ for $\theta\in\dCIR$, and~$\dCIR$ has zero (Lebesgue) measure in~$\overline{\CIR}$. Since $\Wu(\DSS)\cap\dCIR\neq\emptyset$ and $\Ws(\DDS)\cap\dCIR=\emptyset$, the equilibrium~$\DSS$ cannot be complete. Parametrize a side of~$\dCIR$ by $\chi\in[0, 2\pi]$ (for example $\chi=\psi_1=\psi_2$) where $\chi\in\sset{0,2\pi}$ corresponds to~$\Sync\in\dCIR$; cf.~Figure~\ref{fig:DpS}. The dynamics are given by
\begin{align*}
\dot\chi &= \cos(\chi)-1+3r\sin(6\chi) + 3\K\sin(\chi)\\
&=2\sin\!\Big(\frac{\chi}{2}\Big)\left(-2\sin\!\Big(\frac{\chi}{2}\Big)+rT(\chi)+6\K\cos\!\Big(\frac{\chi}{2}\Big)\right)
\end{align*}
where~$T(\chi)$ is a trigonometric polynomial. Since $\K>0$ and~$r$ sufficiently small by assumption, there are exactly three equilibria~$\xi^\DpS$ on~$\dCIR$ (which lie in the same group orbit) with $\chi\approx 2\arctan(3\K)$. These are of saddle type, attracting within~$\dCIR$ and transversely repelling (within~$\DpS$); these are shown in Figure~\ref{fig:DpS}(c). Therefore, 
\[\Wu(\DSS)\cap \Ws(\DDS) = \CIR\sm\Gamma \Wu(\xi^\DpS)\]
is two-dimensional and of full measure in~$\Wu(\DSS)$. This implies that~$\DSS$ is almost complete.

Second, by evaluating the dynamics on~$\pDS$ one can show by a similar argument that there are three equilibria~$\xi^\pDS$ on~$\dCIR$. These are attracting transversely to~$\dCIR$ and repelling within~$\dCIR$. Consequently,
\[\Wu(\DDS)\cap \Ws(\SDS) = \CIR\sm\Gamma \Ws(\xi^\pDS)\]
and~$\DDS$ is almost complete. However, $\DDS$ is not complete since $\Ws(\xi^\pDS)\subset \Wu(\DDS)$ is not contained in~$\Ws(\SDS)$.
 
Finally, the closure of~$\Hf(\Cyc_3)$ now contains the complete heteroclinic cycle
\[\Ccl=(\DSS, \xi^\DpS, \DDS, \xi^\pDS, \SDS, \xi^\SDp, \SDD, \xi^\SpD, \SSD, \xi^\pSD, \DSD, \xi^\DSp).\]
The heteroclinic chain~$\Hf(\Ccl)$ associated with~$\Ccl$ contains~$\Hf(\Cyc_3)$ but is not cyclic nor equable. Indeed, the graph~$\tilde\Gf(\Ccl)$ for~$\Hf(\Ccl)$ contains the noncyclic subgraph depicted in Figure~\ref{fig:HetChain}. This completes the proof of the assertion.
\end{proof}

\section{Dynamics of Networks with Noise and Broken Symmetry}
\label{sec:Numerics}

\newcommand{\Xsk}{X_{\sigma,k}}

\newcommand{\sSSD}{\underline{\Sp\Sp}\Dp}
\newcommand{\sSDS}{\underline{\Sp}\Dp\underline{\Sp}}
\newcommand{\sDSS}{\Dp\underline{\Sp\Sp}}

\newcommand{\dsym}{\delta_\mathrm{sym}}
\newcommand{\dasym}{\delta_\mathrm{asym}}

The heteroclinic cycles lead to switching between localized frequency synchrony which are observed in numerical simulations. First, define the Kuramoto order parameter of population~$\sigma$ as 
$Z_\sigma = \frac{1}{\maxdim}\sum_{j=1}^\maxdim\exp(i\theta_{\sigma,k})$
and let $R_\sigma=\abs{Z_\sigma}$. In particular, $R_\sigma\in[0,1]$ encodes the level of synchrony in each population, that is, $R_\sigma = 1$ iff $\theta_\sigma\in\Sp$ and $R_\sigma = 0$ if $\theta_\sigma\in\Sp$. 
Write~\eqref{eq:DynMxN} as $\dot\theta_{\sigma,k} = X_{\sigma,k}(\theta)$, setting $\omega=-(\maxdim-1)\gt(0)$ without loss of generality such that oscillators appear stationary if phase synchronized in the absence of interpopulation coupling.
Let~$W_{\sigma,k}$ be independent Wiener processes with mean zero and variance one. Since attracting heteroclinic cycles show exponential slowing down of transition times between subsequent saddles, we solve the stochastic differential equation
\begin{equation}\label{eq:DynMxNSimX}
\dot\theta_{\sigma,k} = \Xsk(\theta) + \eta W_{\sigma,k},
\end{equation}
for $\eta>0$ using \textsc{XPP}~\cite{Ermentrout2002}. As shown in Figure~\ref{fig:SimHetCyle}, the dynamics exhibit switching between localized frequency synchrony: populations sequentially accelerate and decelerate as the populations synchronize and desynchronize. The transition times scale with the noise strength as expected~\cite{Stone1999}.

\begin{figure}
\centerline{
\includegraphics[scale=\imagescaling]{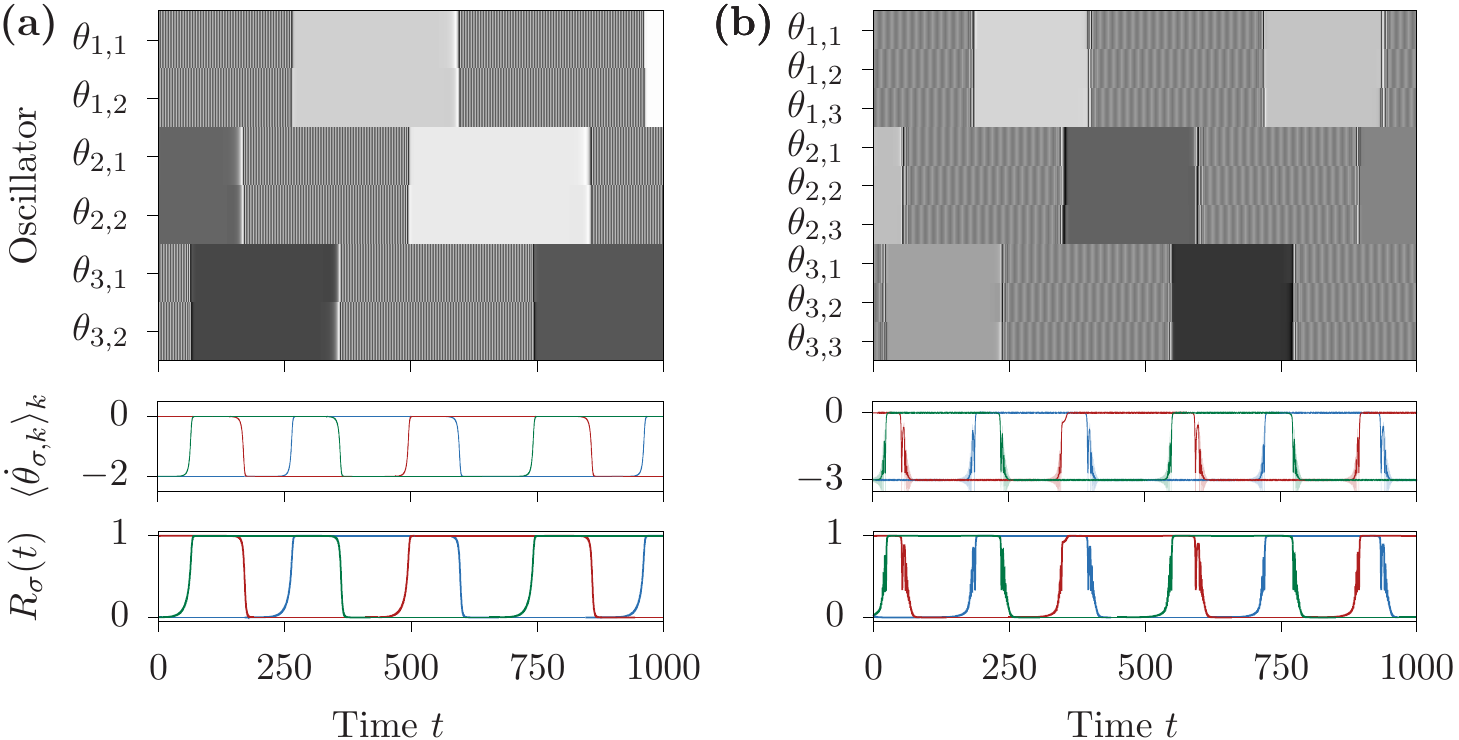}
}
\caption{\label{fig:SimHetCyle}
The heteroclinic cycles~$\Cyc_\maxdim$ induce switching of localized frequency synchrony in the oscillator network~\eqref{eq:DynMxNSimX} of $\maxpop=3$ populations with $\K=0.15$, and noise strength $\eta=10^{-5}$.
Panel~(a) shows the dynamics of $\maxdim=2$ oscillators for $r=0.05$.
Panel~(b) shows the dynamics of $\maxdim=3$ oscillators for $r=0.01$.
The top plot shows the phase evolution of each oscillator (shading indicates the phase where black indicates $\theta_{\sigma,k}=\pi$ and white $\theta_{\sigma,k}=0$).
The middle plot shows the average instantaneous frequencies (lines) as well as the maximum and minimum (shading) per population (colors).
The bottom plot shows the time evolution of the absolute value of the order parameter $R_\sigma$ over time (colors).
Because of the choice of co-rotating frame, the oscillators appear static when they are synchronized; note that different synchronized populations are not necessarily synchronized to each other. A ``kink'' in the order parameter dynamics in Panel~(b)---for example at $t\approx 350$---indicates that the trajectory is passing by an equilibrium on~$\dCIR$.
}
\end{figure}

From the point of view of a phase reduction of a general network of nonlinear oscillators, the interaction terms in the phase oscillator network~\eqref{eq:DynMxN} are nongeneric. Indeed, one would expect that nontrivial pairwise interaction terms would not only be present within populations but also between populations. Here, we asses the effect of forced symmetry breaking on the dynamics~\eqref{eq:DynMxN} in numerical simulations. More specifically, define
\begin{align*}
Y^\textrm{sym}_{\sigma, k}(\theta) &= \frac{1}{\maxpop\maxdim}\sum_{\tau=1}^{\maxpop}\sum_{j=1}^\maxdim\sin(\theta_{\tau,j}-\theta_{\sigma,k}).
\intertext{%
The function $Y^\textrm{sym}_{\sigma, k}$ is~$\Sk{\maxpop\maxdim}$ equivariant and yields pairwise interactions between oscillators between different populations. Let $\iota(\sigma,k) = (\sigma-1)\maxdim+k$ be a linear indexing for all oscillators. Define}
Y^\textrm{asym}_{\sigma, k}(\theta) &= \frac{1}{\maxpop\maxdim}\sin(\theta_{\iota^{-1}(\iota(\sigma,k)+1)} - \theta_{\sigma,k})
\end{align*}
which yields additional pairwise interaction terms. For parameters $\dsym$, $\dasym$ we now consider the evolution of
\begin{equation}\label{eq:DynMxNSim}
\dot\theta_{\sigma,k} = \Xsk(\theta) + \dsym Y^\textrm{sym}_{\sigma, k}(\theta) + \dasym Y^\textrm{asym}_{\sigma, k}(\theta) + \eta W_{\sigma,k}.
\end{equation}

Heteroclinic switching dynamics persist if the phase-shift symmetries are broken. For $\dsym>0$, $\dasym=0$, the $(\Sn\times\Tor)^\maxpop\rtimes \Zm$ symmetry of~\eqref{eq:DynMxN} is broken; if $\G=(\Sn^\maxpop\rtimes\Zm)\times\Tor$ then~\eqref{eq:DynMxNSim} is $\G$-equivariant. In other words, rather than having a phase-shift symmetry for each population, the system~\eqref{eq:DynMxNSim} has a single phase-shift symmetry which acts by adding a constant phase to all $\maxpop\maxdim$ oscillators. While this breaks the invariant subspace structure that gave rise to the robust heteroclinic cycles, we expect certain normally hyperbolic tori to persist for $\dsym>0$ sufficiently small. More precisely, if
\[\sDSS = \tset{(\theta_1, \theta_2, \theta_3)\in\Tormn}{\theta_1\in\Dp, \theta_2,\theta_3\in\Sp, \theta_2=\theta_3}\subset\DSS\] 
(and $\sSDS$, $\sSSD$ are its images under the~$\Zm$ symmetry action) then we expect that invariant tori exist close to~$\sDSS$, $\sSDS$, and $\sSSD$. Indeed, solving the system numerically---as shown in Figure~\ref{fig:SimHetCyle_sb}---indicates that there is in fact a residual attracting heteroclinic network which approaches $\sDSS$, $\sSDS$, and~$\sSSD$.

\begin{figure}
\centerline{
\includegraphics[scale=\imagescaling]{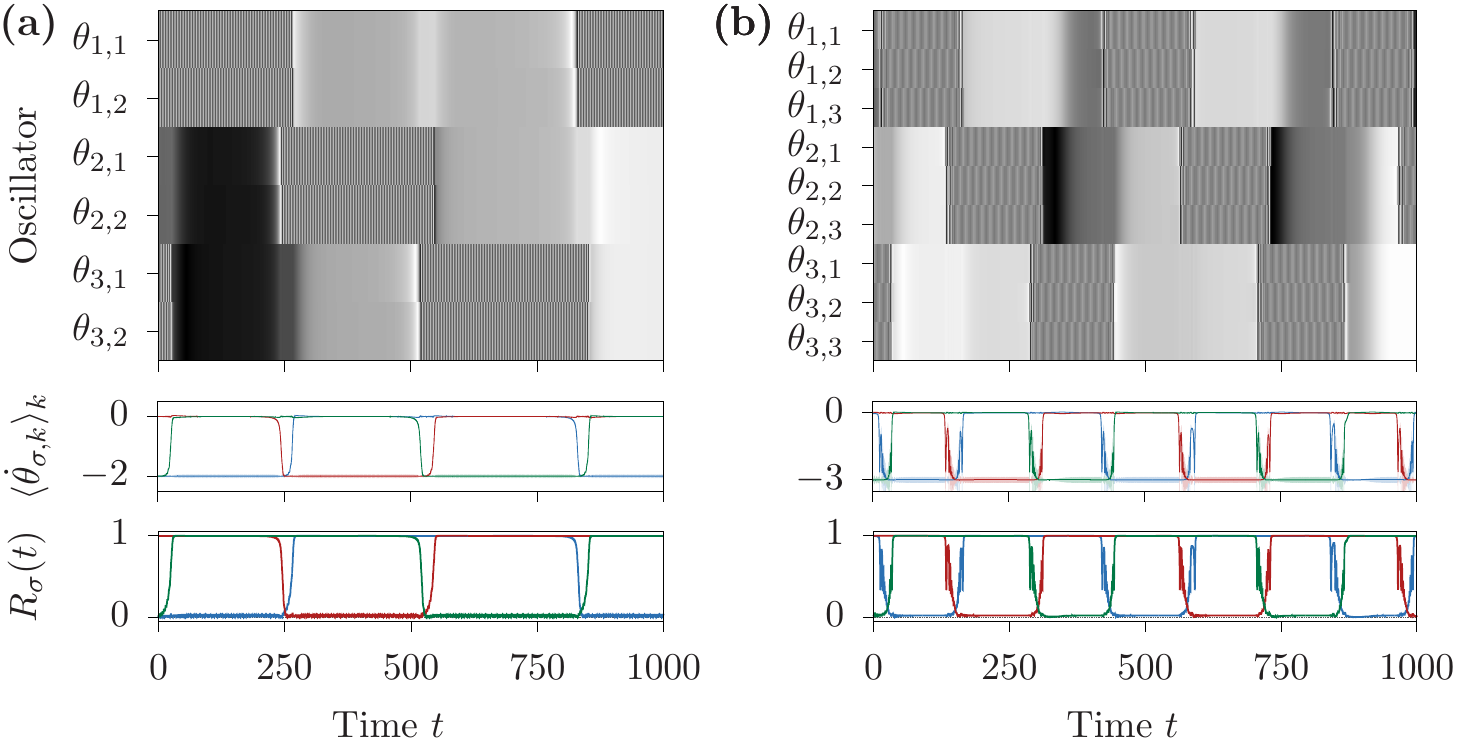}
}
\caption{\label{fig:SimHetCyle_sb}
Switching between localized frequency synchrony persists for the network~\eqref{eq:DynMxNSim} as the phase shift symmetries are broken, $\dsym=0.1$; the other parameters are as in Figure~\ref{fig:SimHetCyle}. Due to the attractive coupling between populations, the synchronized populations now synchronize in phase with each other. In other words, trajectories approach $\sDSS$, $\sSDS$, and $\sSSD$ cyclically.
}
\end{figure}

\begin{figure}
\centerline{
\includegraphics[scale=\imagescaling]{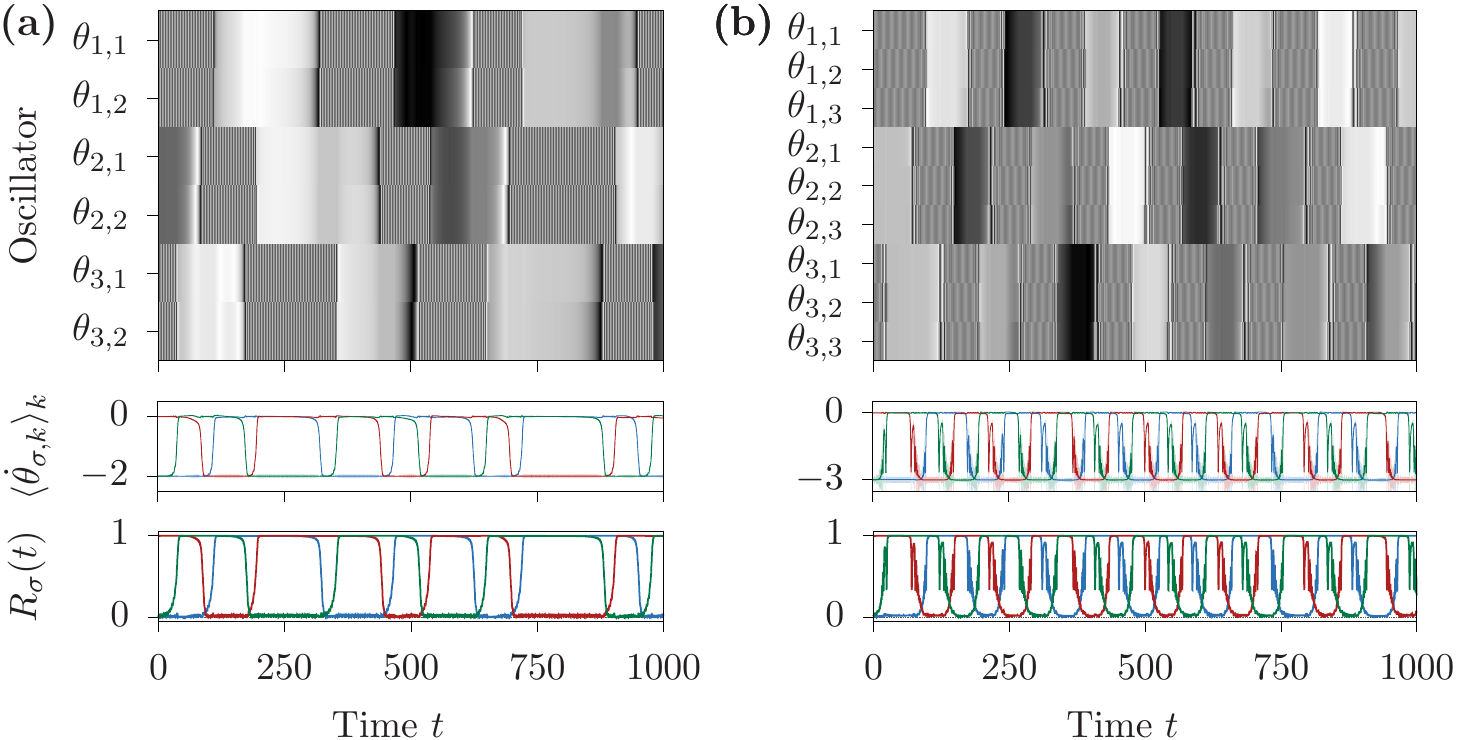}
}
\caption{\label{fig:SimHetCyle_sb_c} Further symmetry breaking leads to deterministic, irregular cycling of localized frequency synchrony for the dynamics of~\eqref{eq:DynMxNSim}. Here $\dsym=0.1$ and $\dasym=0.3\dsym$ while noise is absent, $\eta=0$; all other parameters are as in Figure~\ref{fig:SimHetCyle}.}
\end{figure}

With further forced symmetry breaking, $\dsym, \dasym>0$, the phase oscillator network~\eqref{eq:DynMxNSim} exhibits irregular switching of localized frequency synchrony even in the absence of noise. These potentially chaotic dynamics arise close to the heteroclinic networks~$\Cyc_\maxdim$, $\maxdim=2,3$, as shown in Figure~\ref{fig:SimHetCyle_sb_c}.

\section{Discussion and Conclusions}

Phase oscillator networks with higher-order interactions can give rise to heteroclinic cycles between frequency synchrony; in numerical simulations, these lead to sequential acceleration and deceleration of oscillator populations. Indeed, because of dissipativity, we expect that the attractor of the deterministic system is a subset of the closure of the associated heteroclinic chain. For networks of $\maxdim=2$ oscillators in each population we calculate the stability of the heteroclinic cycles and their bifurcations explicitly in the companion paper~\cite{Bick2018}. For $\maxdim=3$ the unstable manifold of each saddle is (at least) two-dimensional and the assumptions to apply existing stability results~\cite{Krupa1995, Ashwin1998a} are not satisfied. We will address this question in future research.

Rather than assuming weak coupling between populations, the results presented here rely on the symmetries induced by the nonpairwise higher-order network interaction terms. Our numerical simulations for nearby vector fields where these symmetries were broken indicated the persistence of some residual heteroclinic structure. In this context, it would be desirable to extend the methods of forced symmetry breaking~\cite{Sandstede1995,Chossat1995,Guyard1999} to understand the bifurcation behavior for nearby network vector fields with generic interactions.

Numerical simulations indicate that switching dynamics between localized frequency synchrony also arises in networks with $\maxpop=3$ populations of $\maxdim>3$ phase oscillators~\cite{Bick2017c}. Indeed, the methods used here are likely applicable to such networks as well: without higher harmonics, $r=0$, the oscillators are sinusoidally coupled and the phase space~$\Tormn$ is foliated by low-dimensional manifolds~\cite{Pikovsky2011, Chen2017} on which we expect to have a similar potential functions as in the proof of Lemma~\ref{lem:N3SaddleConnection}. While we cannot expect hyperbolicity in this limit due to the degeneracy in the system, suitable network interaction terms with higher harmonics, combined with an approach similar to~\cite{Vlasov2016}, could give rigorous results to show the existence of heteroclinic networks.

In summary, heteroclinic switching dynamics between localized frequency synchrony may arise in networks of identical phase oscillators with higher-order interactions. Here we gave rigorous results for small oscillator networks, but we anticipate similar approaches to be viable for larger networks. While the heteroclinic switching observed here are distinct from those discussed in~\cite{Komarov2011}, where large networks ($\maxdim\geq 1000$) of nonidentical oscillators are considered, it would be interesting to relate the two.

\section*{Acknowledgements}

\noindent The author is indebted to P~Ashwin, JSW~Lamb, A~Lohse, and M~Rabinovich for helpful conversations.


\bibliographystyle{unsrt}
\def\urlprefix{}
\def\url#1{}

\bibliography{/home/shared/Papers/library} 

\end{document}